\documentclass[11pt]{article}
\usepackage{amsmath}
\usepackage{amssymb}
\usepackage{amsthm}
\usepackage{multirow}
\usepackage{hyperref}

\newtheorem{theorem}{Theorem}[section]
\newtheorem{proposition}[theorem]{Proposition}
\newtheorem{corollary}[theorem]{Corollary}

\newtheorem*{urem}{Remark}
\newtheorem{lem}[theorem]{Lemma}
\newtheorem{claim}[theorem]{Claim}

\begin{document}

\title{On degree anti-Ramsey numbers}

\author{Shoni Gilboa \thanks{Mathematics Department, The Open University of Israel, Raanana 43107, Israel. Email: \texttt{tipshoni@gmail.com}}
\and
Dan Hefetz \thanks{Department of Computer Science, Hebrew University, Jerusalem 91904, Israel. Email: \texttt{danny.hefetz@gmail.com}.
      } 
}

\maketitle

\begin{abstract} 
The degree anti-Ramsey number $AR_d(H)$ of a graph $H$ is the smallest integer $k$ for which there exists a graph $G$ with maximum degree at most $k$ such that any proper edge colouring of $G$ yields a rainbow copy of $H$. In this paper we prove a general upper bound on degree anti-Ramsey numbers, determine the precise value of the degree anti-Ramsey number of any forest, and prove an upper bound on the degree anti-Ramsey numbers of cycles of any length which is best possible up to a multiplicative factor of $2$. Our proofs involve a variety of tools, including a classical result of Bollob\'as concerning cross intersecting families and a topological version of Hall's Theorem due to Aharoni, Berger and Meshulam.  
\end{abstract}

\begin{quote}
\end{quote}

\section{Introduction} \label{intro}
 
A copy of a graph $H$ in an edge coloured graph $G$ is called \emph{rainbow} if all edges of $H$ have distinct colours. The \emph{degree anti-Ramsey number} $AR_d(H)$ of a graph $H$ is the smallest integer $k$ for which there exists a graph $G$ with maximum degree at most $k$ such that any proper edge colouring of $G$ yields a rainbow copy of $H$. This notion, which is the focus of this paper, was introduced in~\cite{Alon}.  

Several versions of \emph{anti-Ramsey numbers} appear in the literature (see, e.g.~\cite{FMO} and the many references therein). The \emph{local anti-Ramsey number} $AR(H)$ of a graph $H$ is the smallest integer $n$ such that any proper edge colouring of $K_n$ yields a rainbow copy of $H$. This graph invariant was studied by various researchers, including Babai~\cite{Babai} and Alon, Lefmann and R\"odl~\cite{ALR}. As noted in~\cite{Alon}, it is evident that 
\begin{equation} \label{eq:local}
AR_d(H) \leq AR(H) - 1 \textrm{ holds for any graph } H. 
\end{equation}  

The \emph{size anti-Ramsey} number $AR_s(H)$ of a graph $H$ is the smallest integer $m$ for which there exists a graph $G$ with $m$ edges such that any proper edge colouring of $G$ yields a rainbow copy of $H$. This graph invariant was introduced by Axenovich, Knauer, Stumpp and Ueckerdt~\cite{AKSU} who proved upper and lower bounds on the size anti-Ramsey numbers of paths, cycles, matchings and cliques. In~\cite{Alon}, Alon proved that $AR_d(K_k) = \Theta(k^3/\log k)$ and used this result to prove that $AR_s(K_k) = \Omega(k^6/\log^2 k)$, thus settling a problem of Axenovich et al~\cite{AKSU}. 

It readily follows from~\eqref{eq:local} that any upper bound on $AR(H)$ immediately translates to an upper bound on $AR_d(H)$. One such upper bound was proved by Alon, Jiang, Miller and Pritikin in~\cite{Alon_et_al}. It was proved there that for every graph $H$,
\begin{equation} \label{eq:Alon_et_al}
AR(H) \leq 2 \Delta(H)^2 v(H) + 32 \Delta(H)^4 + 4 v(H).
\end{equation}
 
Our first result is the following improvement:  
 
\begin{theorem} \label{propos:general}
Let $H$ be a graph and let $k$ be its degeneracy. Then
$$
AR(H) \leq k e(H) - k + v(H) \,.
$$
\end{theorem} 

Note that Theorem~\ref{propos:general} is indeed an improvement of~\eqref{eq:Alon_et_al}, since if $H$ is a $k$-degenerate graph, then
$$
k e(H) - k + v(H) \leq k^2 v(H) - k \binom{k+1}{2} - k + v(H) < 2 \Delta(H)^2 v(H) + 32 \Delta(H)^4 + 4 v(H).
$$

Using~\eqref{eq:local} we obtain the following immediate consequence of Theorem~\ref{propos:general}:
\begin{corollary} \label{cor:general}
Let $H$ be a graph and let $k$ be its degeneracy. Then
$$
AR_d(H) \leq k e(H) - k + v(H) - 1 \,.
$$
\end{corollary} 

As observed in~\cite{Alon}, it readily follows from Vizing's Theorem~\cite{Vizing} that
\begin{equation} \label{eq:lower} 
AR_d(H) \geq e(H) - 1 \textrm{ holds for any graph } H.
\end{equation}
It was also observed in~\cite{Alon} that~\eqref{eq:lower} is tight whenever $H$ is a matching with at least $3$ edges (it is obvious that $AR_d(K_2) = 1$ and easy to see that $AR_d(2 K_2) = 2$). Moreover, it was noted in~\cite{Alon} that \eqref{eq:lower} is almost tight for forests, i.e., $AR_d(H) \leq e(H)$ whenever $H$ is a forest. Our next result determines the precise value of the degree anti-Ramsey number of every forest.

\begin{theorem} \label{thm:Forest} 
Let $F$ be a forest. Then $AR_d(F) = e(F) - 1$, unless $F$ is a star of any size or a matching with precisely two edges, in which case $AR_d(F) = e(F)$.
\end{theorem}

Finally, we study degree anti-Ramsey numbers of cycles. It readily follows from~\eqref{eq:lower} and Corollary~\ref{cor:general} that $k-1 \leq AR_d(C_k) \leq 3(k-1)$ holds for every $k \geq 3$. Our next result improves the upper bound.  

\begin{theorem} \label{thm:cycles}
For every $k \geq 3$,
$$
AR_d(C_k) \leq \begin{cases} 2(k-1) & \textrm{if } k \textrm{ is even} \\ 2(k+2) & \textrm{if } k \textrm{ is odd.} \end{cases}
$$
\end{theorem}

For some small values of $k$ we can prove sharper bounds. It is obvious that $AR_d(C_3) = 2$ and we can prove that $AR_d(C_5) \leq 6$ (this will be discussed in Section~\ref{sec::open}). Our next result determines the exact value of $AR_d(C_4)$.  

\begin{proposition} \label{propos:C4}
$$
AR_d(C_4) = 4 \,.
$$
\end{proposition}

\subsection{Notation}
Our graph-theoretic notation is standard and follows that of~\cite{West}. In particular, we use the following. For a graph $G$, let $V(G)$ and $E(G)$ denote its sets of vertices and edges respectively, and let $v(G) = |V(G)|$ and $e(G) = |E(G)|$. For a graph $G$ and a set $S \subseteq V(G)$ we denote by $G[S]$ the graph which is induced on the set $S$.
Denote the maximum degree of a graph $G$ by $\Delta(G)$, and its minimum degree by $\delta(G)$.
For any integer $k \geq 3$, we denote the cycle on $k$ vertices by $C_k$. The \emph{length} of a cycle is the number of its edges. The \emph{girth} of $G$ is the length of a shortest cycle in $G$ (if $G$ is a forest, then its girth is defined to be infinity). The \emph{degeneracy} of a graph $G$ is the smallest integer $d$ for which there exists an ordering $v_1, \ldots, v_n$ of the vertices of $G$ such that $|\{1 \leq j < i : v_j v_i \in E(G)\}| \leq d$ holds for every $1 \leq i \leq n$. A graph with degeneracy $d$ is said to be \emph{$d$-degenerate}. 

The rest of this paper is organized as follows. In Section~\ref{sec:general} we prove Theorem~\ref{propos:general}. In Section~\ref{sec:forests} we prove Theorem~\ref{thm:Forest}. In Section~\ref{sec:cycles} we prove Theorem~\ref{thm:cycles} and Proposition~\ref{propos:C4}. Finally, in Section~\ref{sec::open} we present some open problems.

\section{General upper bound} \label{sec:general}

\begin{proof} [Proof of Theorem~\ref{propos:general}]
Since $H$ is $k$-degenerate, there exists an ordering $u_1, u_2, \ldots, u_{v(H)}$ of the vertices of $H$ such that $|I_{\ell}| \leq k$ holds for every $1 \leq \ell \leq v(H)$, where $I_{\ell} := \{1 \leq i < \ell : u_i u_{\ell} \in E(H)\}$. Let $n = k e(H) - k + v(H)$ and let $c$ be a proper edge colouring of $K_n$. Inductively, we choose vertices $w_1, w_2, \ldots, w_{v(H)}$ in $K_n$ such that the colours of all edges in the set $E_{\ell} := \{w_i w_j : 1 \leq i < j \leq \ell \textrm{ and } u_i u_j \in E(H)\}$ are distinct for every $1 \leq \ell \leq v(H)$. Since, in particular, this is true for $E_{v(H)}$, by mapping $u_i$ to $w_i$ for every $1 \leq i \leq v(H)$, we obtain a rainbow copy of $H$. The vertex $w_1$ may be chosen arbitrarily. Let $2 \leq \ell \leq v(H)$ and assume the vertices $w_1, w_2, \ldots, w_{\ell-1}$ were already chosen and we now wish to choose $w_{\ell}$. Let $W_{\ell} = V(K_n) \setminus \{w_1, w_2, \ldots, w_{\ell-1}\}$. For every $i \in I_{\ell}$ we have 
$$
|\{w \in W_{\ell} : \exists e \in E_{\ell-1} \textrm{ such that } c(w_i w) = c(e)\}| \leq |E_{\ell-1}| \leq e(H) - |I_{\ell}|
$$ 
and therefore
\begin{eqnarray*}
|\{w \in W_{\ell} : \exists i \in I_{\ell}, \; \exists e \in E_{\ell-1} \textrm{ such that } c(w_i w) = c(e)\}| &\leq& |I_{\ell}| \left(e(H) - |I_{\ell}|\right) \leq k e(H) - k \\
&<& n - (\ell-1) = |W_{\ell}| \,.
\end{eqnarray*}
where the last inequality holds by our choice of $n$.

We can choose $w_{\ell}$ to be any vertex of the non-empty set $W_{\ell} \setminus \{w \in W_{\ell} : \exists i \in I_{\ell}, \; \exists e \in E_{\ell-1} \textrm{ such that } c(w_i w) = c(e)\}$.
\end{proof}

\begin{urem} Local anti-Ramsey numbers are discussed in~\cite{Alon_et_al} as a special case of a broader notion. For a graph $H$ and a positive integer $m$, let $g(m,H)$ be the smallest integer $n$ such that any edge colouring of $K_n$ in which no colour appears more than $m$ times at each vertex yields a rainbow copy of $H$. Obviously $g(1,H)$ is simply $AR(H)$. The upper bound~\eqref{eq:Alon_et_al} is a special case of the more general upper bound 
\begin{equation} \label{eq:g}
g(m,H) \leq 2 m \Delta(H)^2 v(H) + 32 m \Delta(H)^4 + 4 v(H)
\end{equation}
for every graph $H$. The proof of Theorem~\ref{propos:general} extends, with obvious changes, to an upper bound on $g(m,H)$ for every $m$,
$$
g(m,H) \leq m k e(H) - m k + v(H) \leq m k^2 v(H) - m k \binom{k+1}{2} - m k + v(H)
$$
which is an improvement of~\eqref{eq:g}. 
\end{urem}

\section{Degree anti-Ramsey numbers of forests} \label{sec:forests}

\begin{lem} \label{lem:same_colours}
Let $T$ be a tree on $k+2$ vertices which is not a star. Let $G$ be a $k$-regular connected graph with girth at least $k+2$. If there exists a proper edge colouring of $G$ with no rainbow copy of $T$, then $G$ is $k$-edge-colourable.
\end{lem}

\begin{proof} 
Let $x_0$ be some leaf of $T$ and let $x_1$ be its unique neighbour. Let $x_2$ be a neighbour of $x_1$ in $T$ which is not a leaf; such a vertex $x_2$ exists since $T$ is not a star. For $i \in \{1,2\}$, let $T_i$ denote the tree in $T \setminus \{x_0, x_1 x_2\}$ which contains $x_i$. Let $x_1, y_1, \ldots, y_s$ be an ordering of the vertices of $T_1$ such that $T_1[\{x_1, y_1, \ldots, y_j\}]$ is a tree for every $1 \leq j \leq s$. Similarly, let $x_2, z_1, \ldots, z_t$ be an ordering of the vertices of $T_2$ such that $T_2[\{x_2, z_1, \ldots, z_j\}]$ is a tree for every $1 \leq j \leq t$. Note that $t \geq 1$ and that $t + s = k-1$.  

Let $G$ be as in the statement of the lemma and let $c$ be a proper edge colouring of $G$ with no rainbow copy of $T$.
For any vertex $v \in V(G)$ let $N_c(v) = \{c(e) : e \in E(G) \textrm{ and } v \in e\}$. 
Suppose for a contradiction that $G$ is not $k$-edge-colourable, in particular, there exist vertices $v_1, v_2 \in V(G)$ such that $N_c(v_1) \neq N_c(v_2)$. Since $G$ is connected, it follows that there are also adjacent vertices $u_1, u_2 \in V(G)$ such that $N_c(u_2)$ is not contained in $N_c(u_1)$. We will describe an embedding $\phi$ of $T$ into $G$ such that $\phi(T)$ is rainbow; this will contradict our assumption that no copy of $T$ in $G$ is rainbow under $c$. We set $\phi(x_1) = u_1$ and $\phi(x_2) = u_2$. Let $w_1$ be a neighbour of $u_2$ in $G$ such that $c(u_2 w_1) \notin N_c(u_1)$; we set $\phi(z_1) = w_1$. Inductively, for every $2 \leq i \leq t$, we can find a vertex $w_i \in V(G)$ which satisfies the following two properties:
\begin{description}
\item [(1)] $w_i \notin \{u_1, u_2, w_1, \ldots, w_{i-1}\}$.
\item [(2)] $c(w_i \phi(z)) \notin \{c(e) : e \in \{u_1 u_2\} \cup E(\phi(T_2[\{x_2, z_1, \ldots, z_{i-1}\}]))\}$ where $z \in \{x_2, z_1, \ldots, z_{i-1}\}$ is the unique vertex for which $z z_i$ is an edge of $T_2$.   
\end{description}
Such a vertex $w_i$ exists since $i < k$ and since the girth of $G$ is at least $k+2$. For every $2 \leq i \leq t$ we set $\phi(z_i) = w_i$.

Similarly, for every $1 \leq i \leq s$, we can find a vertex $w'_i \in V(G)$ which satisfies the following two properties:
\begin{description}
\item [(1')] $w'_i \notin \{u_1, u_2, w_1, \ldots, w_t, w'_1, \ldots, w'_{i-1}\}$.
\item [(2')] $c(w'_i \phi(y)) \notin \{c(e) : e \in \{u_1 u_2\} \cup E(\phi(T_2)) \cup E(\phi(T_1[\{x_1, y_1, \ldots, y_{i-1}\}]))\}$ where $y \in \{x_1, y_1, \ldots, y_{i-1}\}$ is the unique vertex for which $y y_i$ is an edge of $T_1$.   
\end{description}
Such a vertex $w'_i$ exists since $t + i < k$ and since the girth of $G$ is at least $k+2$. For every $1 \leq i \leq s$ we set $\phi(y_i) = w'_i$. 

Finally, we can find a vertex $u_0 \in V(G)$ such that $u_0 \notin \{u_1, u_2, w_1, \ldots, w_t, w'_1, \ldots, w'_s\}$, $u_0 u_1 \in E(G)$ and $c(u_0 u_1) \notin \{c(e) : e \in E(\phi(T \setminus \{x_0\}))\setminus\{u_2 w_1\}\}$. Such a vertex exists since $|\{c(e) : e \in E(\phi(T \setminus \{x_0\}))\setminus\{u_2 w_1\}\}| < k$ and since the girth of $G$ is at least $k+2$ . Since $c(u_2 w_1) \notin N_c(u_1)$ by assumption, it follows that $c(u_0 u_1) \neq c(u_2 w_1)$ and thus $\phi(T)$ is a rainbow copy of $T$ in $G$.   
\end{proof}

\begin{lem} \label{lem::cutVertex}
Let $k$ be a positive integer and let $G$ be a $k$-regular connected graph. If $G$ has a cut-vertex, then $G$ is not $k$-edge-colourable.
\end{lem}

Lemma~\ref{lem::cutVertex} appears as an exercise in~\cite{West}; for the sake of completeness we include a short proof.

\begin{proof}
Suppose for a contradiction that $G$ is $k$-edge-colourable. Since $G$ is $k$-regular, it follows that $E(G)$ can be decomposed into $k$ perfect matchings $M_1, \ldots, M_k$; in particular $v(G)$ is even. Let $x$ be a cut-vertex of $G$ and let $C_1, C_2, \ldots, C_t$ be the connected components of $G \setminus x$. Since $|V(G) \setminus \{x\}|$ is odd, there must exist some $1 \leq i \leq t$ such that $|C_i|$ is odd. Assume without loss of generality that $|C_1|$ is odd and let $y \in C_2$ be a neighbour of $x$. Assume without loss of generality that $xy \in M_1$. Then the restriction of $M_1$ to $C_1$ must be a perfect matching of $G[C_1]$, which is impossible as $|C_1|$ is odd.  
\end{proof}

\begin{proposition} \label{prop::highGirthClass2}
For every integer $k \geq 2$, there exists a $k$-regular connected graph with girth at least $k+2$ which is not $k$-edge-colourable.
\end{proposition}

\begin{proof}
We distinguish between two cases according to the parity of $k$. Assume first that $k$ is even. Since for $k=2$ we can take $C_5$, we assume that $k \geq 4$. Let $G$ be a $k$-regular connected graph with an odd number of vertices and with girth at least $k+2$. Such a graph $G$ exists since $k$ is even and since, with positive probability, for every $k\geq 3$, a random $k$-regular graph is connected and has arbitrarily large girth (see, e.g.,~\cite{Wormald}). Since $v(G)$ is odd, $G$ cannot be $k$-edge-colourable.     

Next, assume that $k \geq 3$ is odd. Let $H_1, \ldots, H_{k-1}$ be pairwise vertex disjoint $k$-regular connected graphs with girth at least $k+2$. For every $1 \leq i \leq k-1$ let $x_i y_i$ be an arbitrary edge of $H_i$ and let $G_i = H_i \setminus \{x_i y_i\}$. Let $G$ be the graph with vertex set $V(G_1) \cup \ldots \cup V(G_{k-1}) \cup \{u,v\}$ and edge set $E(G_1) \cup \ldots \cup E(G_{k-1}) \cup \{uv\} \cup \{u x_i : 1 \leq i \leq (k-1)/2\} \cup \{u y_i : 1 \leq i \leq (k-1)/2\} \cup \{v x_i : (k-1)/2 + 1 \leq i \leq k-1\} \cup \{v y_i : (k-1)/2 + 1 \leq i \leq k-1\}$. It is evident that $G$ is a $k$-regular connected graph with girth at least $k+2$. Since, moreover, $u$ is clearly a cut-vertex of $G$, it follows by Lemma~\ref{lem::cutVertex} that $G$ is not $k$-edge-colourable.   
\end{proof}

In the proof of Theorem~\ref{thm:Forest} we will make use of the following well-known result of Bollob\'as~\cite{Bollobas}.

\begin{theorem} \label{thm::Bollobas}
Let $a, b$ and $N$ be positive integers and let $A_1, \ldots, A_N, B_1, \ldots, B_N$ be sets satisfying the following three properties:
\begin{description}
\item [(i)] $|A_i| = a$ and $|B_i| = b$ for every $1 \leq i \leq N$.
\item [(ii)] $A_i \cap B_i = \emptyset$ for every $1 \leq i \leq N$.
\item [(iii)] $A_i \cap B_j \neq \emptyset$ for every $1 \leq i \neq j \leq N$.
\end{description}
Then $N \leq \binom{a+b}{a}$.
\end{theorem} 

\begin{proof} [Proof of Theorem~\ref{thm:Forest}]
It is obvious that $AR_d(F) = e(F)$ if $F$ is a star or a matching with precisely two edges. Moreover, as noted in~\eqref{eq:lower}, $AR_d(F) \geq e(F) - 1$ holds for every graph $F$. Hence, in order to complete the proof, we need to show that if $F$ is a forest which is not a star or a matching with precisely two edges, then $AR_d(F) \leq e(F) - 1$. We will do so by induction on $r$, the number of connected components of $F$, i.e., the number of trees in the forest $F$.

We begin by addressing the base case $r=1$. Let $F$ be a tree on $k+2$ vertices which is not a star; note that $k \geq 2$. Let $G$ be a $k$-regular connected graph with girth at least $k+2$ which is not $k$-edge-colourable; such a graph exists by Proposition~\ref{prop::highGirthClass2}. 
%Let $c$ be any proper edge colouring of $G$. 
%According to Lemma~\ref{lem:same_colours}, $G$ must admit a copy of $F$ which is rainbow under $c$ as otherwise $c$ would be a proper $k$-edge-colouring of $G$.     
According to Lemma~\ref{lem:same_colours}, any proper edge colouring of $G$ yields a rainbow copy of $F$, otherwise $G$ would be $k$-edge-colourable.

Now fix $r \geq 1$ and assume that $AR_d(F) \leq e(F) - 1$ holds for every forest $F$ which consists of $r$ trees and is not a star or a matching with precisely two edges. Let $F$ be a forest which consists of $r+1$ trees and is not a matching with precisely two edges. Choose some tree $T$ in $F$ and let $R$ be the forest obtained from $F$ by removing $T$. Let $\tilde{F}$ be a graph which is not a star and is obtained from $F$ by identifying some vertex of $T$ with some vertex of $R$. Such a graph $\tilde{F}$ exists since $F$ is not a matching with precisely two edges. Clearly $\tilde{F}$ is a forest which consists of $r$ trees. Since, moreover, it is not a star or a matching with precisely two edges, it follows by the induction hypothesis that there exists a graph $\tilde{G}$ with maximum degree at most $e(\tilde{F}) - 1 = e(F) - 1$ such that any proper edge colouring of $\tilde{G}$ yields a rainbow copy of $\tilde{F}$. Let $a$ denote the number of edges in $T$, let $b$ denote the number of edges in $R$, and let $N = \binom{a+b}{a} + 1$. Let $G$ be the union of $N$ pairwise vertex disjoint copies $G_1, G_2, \ldots, G_N$ of $\tilde{G}$. The maximum degree of $G$ is clearly at most $e(F) - 1$. Suppose for a contradiction that there exists a proper edge colouring of $G$ that contains no rainbow copy of $F$. For every $1 \leq i \leq N$, the graph $G_i$ contains a rainbow copy $F_i$ of $\tilde{F}$. Let $A_i$ and $B_i$ be the sets of colours of the edges of the $T$ part and the $R$ part of $F_i$, respectively. Surely $A_i \cap B_i = \emptyset$ holds for every $1 \leq i \leq N$ as $F_i$ is rainbow. On the other hand, for every $1 \leq i \neq j \leq N$ we must have $A_i \cap B_j \neq \emptyset$ as otherwise the $T$ part of $F_i$ and the $R$ part of $F_j$ would form a rainbow copy of $F$. We conclude that the sets $A_1, \ldots, A_N, B_1, \ldots, B_N$ satisfy the conditions of Theorem~\ref{thm::Bollobas} and thus $\binom{a+b}{a} + 1 = N \leq \binom{a+b}{a}$ which is obviously a contradiction.   
\end{proof}

\section{Degree anti-Ramsey numbers of cycles} \label{sec:cycles}

Our first goal in this section is to prove Theorem~\ref{thm:cycles}. Our proof uses a topological version of Hall's Theorem due to Aharoni, Berger and Meshulam. In order to state their theorem, we need the following terminology. The \emph{fractional width} $w^*(\mathcal{F})$ of a hypergraph $\mathcal{F}$ is defined to be the minimum of $\sum_{E \in \mathcal{F}} \lambda(E)$ over all functions $\lambda : \mathcal{F} \to [0, \infty)$ such that $\sum_{E \in \mathcal{F}} |E \cap F| \cdot \lambda(E) \geq 1$ holds for every $F \in \mathcal{F}$.

\begin{theorem} \label{theorem:ABM} \cite[Theorem 1.5] {ABM}
If $\{\mathcal{F}_i\}_{i=1}^k$ are hypergraphs satisfying $w^*\left(\bigcup_{i \in I} \mathcal{F}_i\right) > |I| - 1$ for every $I \subseteq [k]$, then $\{\mathcal{F}_i\}_{i=1}^k$ admits a system of disjoint representatives, i.e., pairwise disjoint sets $F_1, F_2, \ldots, F_k$ such that $F_i \in \mathcal{F}_i$ holds for every $1 \leq i \leq k$.
\end{theorem}

\begin{proof}[Proof of Theorem~\ref{thm:cycles}]
We will first consider even cycles, i.e., we will prove that $AR_d(C_{2k}) \leq 2(2k-1)$ holds for every $k \geq 2$. For all positive integers $k \geq 2$ and $d$, let $G_{2k,d}$ denote the graph with vertex set
$$
\{u_i : 1 \leq i \leq k\} \cup \{v_{i,j} : 1 \leq i \leq k, \; 1 \leq j \leq d\} \,,
$$ 
and edge set
$$
\{u_i v_{i,j} : 1 \leq i \leq k, \; 1 \leq j \leq d\} \cup \{v_{i,j} u_{(i\bmod k)+1} : 1 \leq i \leq k, \; 1 \leq j \leq d\} \,.
$$
Since the maximum degree of $G_{2k,d}$ is $2d$, it suffices to prove that, for every $d \geq 2k-1$, any proper edge colouring of $G_{2k,d}$ admits a rainbow copy of $C_{2k}$.  

Fix some $d \geq 2k-1$ and let $c$ be a proper edge colouring of $G_{2k,d}$. Consider $3$-uniform hypergraphs $\mathcal{F}_1, \mathcal{F}_2, \ldots, \mathcal{F}_k$ on the common set of vertices $\{v_{i,j} : 1 \leq i \leq k, \; 1 \leq j \leq d\} \cup \{c(e) : e \in E\left(G_{2k,d}\right)\}$. For every
$1 \leq i \leq k$ and $1 \leq j \leq d$, let $E_{i,j} = \{c(u_i v_{i,j}), v_{i,j}, c(v_{i,j} u_{(i \bmod k) + 1})\}$ and let $\mathcal{F}_i = \{E_{i,j} : 1 \leq j \leq d\}$. For every $I \subseteq [k]$, denote $\mathcal{F}_I := \bigcup_{i \in I} \mathcal{F}_i$. Since $c$ is a proper edge colouring of $G_{2k,d}$, for every $1 \leq i \leq k$ and every colour $\alpha$ we have $|\{e \in E\left(G_{2k,d}\right) : u_i \in e \textrm{ and } c(e) = \alpha\}| \leq 1$. It follows that $|\{F \in \mathcal{F}_I : \alpha \in F\}| \leq k$ holds for every $I \subseteq [k]$ and every colour $\alpha$. Therefore, for every $I \subseteq [k]$ and every $E_{i,j} \in \mathcal{F}_I$ we have
\begin{equation} \label{eq:sum_intersections_even}
\sum_{F \in \mathcal{F}_I} |F \cap E_{i,j}| = 1 + |\{F \in \mathcal{F}_I : c(u_i v_{i,j}) \in F\}| + |\{F \in \mathcal{F}_I : c(v_{i,j} u_{(i \bmod k) + 1}) \in F\}| \leq 2k+1.
\end{equation}

For every $I \subseteq [k]$, let $\lambda_I : \mathcal{F}_I \to [0,\infty)$ be a function such that $\sum_{E \in \mathcal{F}_I} |E \cap E_{i,j}| \cdot \lambda_I(E) \geq 1$ holds for every $E_{i,j} \in \mathcal{F}_I$. Then, it follows by~\eqref{eq:sum_intersections_even} that
\begin{gather*}
d|I| = |\mathcal{F}_I| \leq \sum_{F \in \mathcal{F}_I} \left(\sum_{E \in \mathcal{F}_I} |F \cap E| \cdot \lambda_I(E)\right) = \sum_{E \in \mathcal{F}_I} \left(\sum_{F \in \mathcal{F}_I} |F \cap E|\right) \lambda_I(E) \leq \left(2k+1\right) \sum_{E \in \mathcal{F}_I} \lambda_I(E).
\end{gather*}
Therefore, for every $I \subseteq [k]$ we have 
$$
w^*(\mathcal{F}_I) \geq \frac{d |I|}{2k+1} \geq \frac{(2k-1)|I|}{2k+1} = |I| - \frac{2|I|}{2k+1} > |I| - 1.
$$ 

It thus follows by Theorem~\ref{theorem:ABM} that $\{\mathcal{F}_i\}_{i=1}^k$ admits a system of disjoint representatives, i.e., pairwise disjoint sets $E_{1,j_1}, E_{2,j_2}, \ldots, E_{k,j_k}$. We conclude that $u_1 v_{1,j_1} u_2 v_{2,j_2} \ldots u_k v_{k,j_k} u_1$
form a rainbow copy of $C_{2k}$.

Next, we will consider odd cycles, i.e., we will prove that $AR_d(C_{2k-1}) \leq 2(2k+1)$ holds for every $k \geq 2$. For all positive integers $k \geq 2$ and $d$, let $G_{2k-1,d}$ denote the graph with vertex set
$$
\{u_i : 1 \leq i \leq k\} \cup \{v_{i,j} : 1 \leq i \leq k-1, \; 1 \leq j \leq d\} \,,
$$ 
and edge set
$$
\{u_i v_{i,j} : 1 \leq i \leq k-1, \; 1 \leq j \leq d\} \cup \{v_{i,j} u_{i+1} : 1 \leq i \leq k-1, \; 1 \leq j \leq d\} \cup \{u_k u_1\} \,.
$$
Since the maximum degree of $G_{2k-1,d}$ is $2d$, it suffices to prove that, for every $d \geq 2k+1$, any proper edge colouring of $G_{2k-1,d}$ admits a rainbow copy of $C_{2k-1}$. Fix some $d \geq 2k+1$, let $c$ be a proper edge colouring of $G_{2k-1,d}$, and let $\alpha = c(u_k u_1)$. Let $H$ denote the graph obtained from $G_{2k-1,d}$ by removing the edge $u_k u_1$ and every vertex $v_{i,j}$ for which $\alpha \in \{c(u_i v_{i,j}), c(v_{i,j} u_{i+1})\}$. 
Let $H'$ be obtained from $H$ by adding new vertices $w_1, w_2, \ldots, w_{d-2}$ and new edges $u_1 w_1, \ldots, u_1 w_{d-2}, w_1 u_k, \ldots, w_{d-2} u_k$. Extend $c$ arbitrarily to a proper edge colouring of $H'$. 
Since $d-2 \geq 2k-1$ and $G_{2k,d-2} \subseteq H'$, it follows that $H'$ admits a copy $u_1 v_{1,j_1} u_2 v_{2,j_2} \ldots v_{k-1, j_{k-1}} u_k w_s u_1$ of $C_{2k}$ which is rainbow under $c$. We conclude that $u_1 v_{1,j_1} u_2 v_{2,j_2} \ldots v_{k-1, j_{k-1}} u_k u_1$ is a rainbow copy of $C_{2k-1}$ in $G_{2k-1,d}$.   
\end{proof}

Our second goal in this section is to determine the degree anti-Ramsey number of $C_4$. 

\begin{proof}[Proof of Proposition \ref{propos:C4}]
It is easy to verify that every proper edge colouring of $K_{2,4}$ (note that $K_{2,4} \cong G_{4,2}$) yields a rainbow copy of $C_4$ and thus $AR_d(C_4) \leq 4$. Suppose for a contradiction that there exists a graph $G$ with maximum degree at most $3$ such that any proper edge colouring of $G$ yields a rainbow copy of $C_4$; let $G$ be an inclusion minimal such graph.

\begin{claim} \label{cl::minimalIs3regular}
$G$ is bridgeless and $3$-regular. 
\end{claim}

\begin{proof}
By assumption $\Delta(G) \leq 3$. Moreover, it readily follows from the minimality of $G$ that it is bridgeless; in particular, $\delta(G) \geq 2$. It thus remains to prove that $\delta(G) = 3$. Suppose for a contradiction that there exists a vertex of degree $2$ in $G$. Let $u$ be such a vertex and let $v_1, v_2$ be its neighbours. It follows by the minimality of $G$ that $G \setminus \{u\}$ admits a proper edge colouring $c$ with no rainbow copy of $C_4$. In order to obtain a contradiction we will show that we can extend, and modify if needed, the colouring $c$ to a proper edge colouring of $G$ with no rainbow copy of $C_4$. Let $\alpha$ and $\beta$ be two colours not in $\{c(e) : e \in E(G \setminus \{u\})\}$. We distinguish between three cases according to the size of the common neighbourhood of $v_1$ and $v_2$ in $G$.
\begin{description}
\item [Case 1:] $N_G(v_1) \cap N_G(v_2) = \{u\}$. Set $c(u v_1) = \alpha$ and $c(u v_2) = \beta$.

\item [Case 2:] $|N_G(v_1) \cap N_G(v_2)| = 3$. Let $w_1$ and $w_2$ denote the common neighbours of $v_1$ and $v_2$ in $G \setminus \{u\}$. Since the restriction of $c$ to $G \setminus \{u\}$ contains no rainbow copy of $C_4$, there must exist an $i \in \{1,2\}$ such that $c(v_1 w_i) = c(v_2 w_{3-i})$. Change the colour of $v_1 w_{3-i}$ to $\alpha$, colour $u v_1$ by $c(v_2 w_i)$, and set $c(u v_2) = \alpha$.   

\item [Case 3:] $|N_G(v_1) \cap N_G(v_2)| = 2$. Let $w$ denote the unique common neighbour of $v_1$ and $v_2$ in $G \setminus \{u\}$. If there is no copy of $C_4$ in $G \setminus \{u\}$ containing the edge $v_1 w$, then change the colour of $v_1 w$ to $\alpha$ and set $c(u v_2) = \alpha$ and $c(u v_1) = \beta$. Assume then that there exists a copy $v_1 w z_1 z_2 v_1$ of $C_4$ in $G \setminus \{u\}$ containing the edge $v_1 w$; note that there is exactly one such copy. Since $G \setminus \{u\}$ contains no rainbow copy of $C_4$, either $c(v_1 w) = c(z_1 z_2)$ or $c(v_1 z_2) = c(w z_1)$. If $c(v_1 w) = c(z_1 z_2)$, then change the colour of $v_1 z_2$ to $\alpha$, colour $u v_1$ by $c(v_2 w)$, and set $c(u v_2) = \beta$. If $c(v_1 z_2) = c(w z_1)$, then change the colour of $v_1 w$ to $\alpha$ and set $c(u v_2) = \alpha$ and $c(u v_1) = \beta$.   
\end{description}  
\end{proof}
 
Since, by Claim~\ref{cl::minimalIs3regular}, $G$ is a bridgeless cubic graph, it follows by Petersen's Theorem~\cite{Petersen} that $G$ admits a perfect matching $M$. Let $H$ be the graph obtained from $G$ by removing the edges of $M$, then $H$ is the disjoint union of cycles.   

An edge $v_i v_{i+1}$ of an odd cycle $v_0 v_1 \ldots v_{2k} v_0$ of $H$ is said to be \emph{free} if $\{v_{i-2} v_{i+1}, v_{i-1} v_{i+2}, v_i v_{i+3}\} \cap M = \emptyset$, where addition is taken modulo $2k+1$. Equivalently, $v_i v_{i+1}$ is free if any copy of $C_4$ in $G$ which includes $v_i v_{i+1}$ must contain two edges of $M$.  

\begin{claim} \label{cl::OddCycleWithChords}
Every odd cycle of $H$ has at least one free edge.
\end{claim}

\begin{proof}
Let $v_0 v_1 \ldots v_{2k} v_0$ be an odd cycle of $H$ (if no such cycle exists, then there is nothing to prove). Let $I = \{0 \leq i \leq 2k : v_i v_{i+3} \in M\}$, where addition is taken modulo $2k+1$. Since $M$ is a matching, it follows that $|\{i, i+3\} \cap I| \leq 1$ for every $0 \leq i \leq 2k$. Therefore $|I| \leq \left\lfloor \frac{2k+1}{2} \right\rfloor = k$ and thus there exists an $0 \leq i \leq 2k$ such that $\{i+1, i+2\} \cap I = \emptyset$. If $i\in I$, then $i+3\notin I$ and therefore the edge $v_{i+3} v_{i+4} $ is free. Otherwise, the edge $v_{i+2}v_{i+3}$ is free.
\end{proof}

The graph obtained from $H$ by removing one free edge from each odd cycle is clearly $2$-edge-colourable. We extend this colouring to a proper edge colouring $c$ of $G$ by colouring the removed free edges by a third colour and the edges of $M$ by a fourth colour. If $F$ is a copy of $C_4$ which is rainbow under $c$, then it must contain exactly one free edge and exactly one edge of $M$. However, no such copies exist by the definition of free edges. This contradicts our assumption that any proper edge colouring of $G$ admits a rainbow copy of $C_4$. 
\end{proof}

\section{Concluding remarks and open problems} \label{sec::open}

\noindent \textbf{Graphs $H$ for which $AR_d(H) = e(H) - 1$.} As noted in the introduction, $AR_d(H) \geq e(H) - 1$ for every graph $H$. In Theorem~\ref{thm:Forest} this simple lower bound is shown to be tight for every forest which is not a star nor a matching with precisely two edges. It would be interesting to characterize the family ${\mathcal F}$ of all graphs $H$ for which $AR_d(H) = e(H) - 1$. One should note that this family is quite rich. In particular, it is not hard to see that, for every graph $H$, there exists an integer $t_0$ such that, for every $t \geq t_0$, the graph $H_t$ which is the disjoint union of $H$ and a matching of size $t$ is in ${\mathcal F}$. Indeed, let $r = AR_d(H)$ and let $G$ be a graph with maximum degree $r$ such that any proper edge colouring of $G$ yields a rainbow copy of $H$. Let $t_0 =  r + 1 - e(H)$ and fix some $t \geq t_0$. For any $1 \leq i \leq t$ let $F_i$ be an arbitrary $(t + e(H) - 1)$-regular Class 2 graph and let $G_t$ be the pairwise vertex disjoint union of $F_1, \ldots, F_t$ and $G$. It is easy to see that any proper edge colouring of $G_t$ yields a rainbow copy of $H_t$ and thus $AR_d(H_t) = e(H_t) - 1$. 

In light of this example, it might be easier to study the sub-family ${\mathcal F}_C \subseteq {\mathcal F}$ of connected graphs $H$ for which $AR_d(H) = e(H) - 1$. Even in this case, we have a few examples of graphs in ${\mathcal F}_C$, apart from $C_3$ and trees which are not stars. Indeed, let $G_1$ be the triangle with a pendant edge (i.e. the graph with vertex set $\{x_1, x_2, x_3, x_4\}$ and edge set $\{x_1 x_2, x_2 x_3, x_1 x_3, x_3 x_4\}$) and let $G_2$ be the ``bull'' (i.e. the graph with vertex set $\{x_1, x_2, x_3, x_4, x_5\}$ and edge set $\{x_1 x_2, x_2 x_3, x_1 x_3, x_2 x_4, x_3 x_5\}$). It is not hard to verify that any proper edge colouring of the graph obtained from $K_4$ by subdividing one of its edges once yields a rainbow copy of $G_1$ and that any proper edge colouring of $K_5$ yields a rainbow copy of $G_2$.         

\bigskip

\noindent \textbf{Degree anti-Ramsey numbers of cycles.} It follows from the general lower bound~\eqref{eq:lower} that $AR_d(C_k) \geq k-1$ for every $k \geq 3$. This is tight for $k = 3$ but not for $k = 4$ as $AR_d(C_4) = 4$ by Proposition~\ref{propos:C4}. The latter result is the only non-trivial lower bound we have on the degree anti-Ramsey number of any cycle. It would be interesting to prove non-trivial lower bounds on $AR_d(C_k)$ for $k \geq 5$.   

A general non-trivial upper bound on $AR_d(C_k)$ is stated in Theorem~\ref{thm:cycles}. It is proved there that $AR_d(C_k) \leq 2(k-1)$ whenever $k$ is even and that $AR_d(C_k) \leq 2(k+2)$ whenever $k$ is odd. It would be interesting to determine $AR_d(C_k)$ for every $k \geq 5$ or at least to significantly reduce the gap between the lower and upper bounds.  

All the upper bounds on $AR_d(C_k)$ we proved in this paper were obtained by examining the proper edge colourings of the graphs $G_{k,d}$ which are defined in the proof of Theorem~\ref{thm:cycles}. 
We believe that finding the smallest $d = d(k)$ for which every proper edge colouring of $G_{k,d}$ yields a rainbow copy of $C_k$ is of independent interest and might be helpful in improving the upper bounds we currently have on $AR_d(C_k)$. It is not hard to verify that $d(3) = 1$, $d(4) = 2$ and $d(5) = 3$ (the latter shows, in particular, that $AR_d(C_5) \leq 6$). Moreover, it was shown in the proof of Theorem~\ref{thm:cycles} that $d(k) \leq k-1$ whenever $k$ is even and that $d(k) \leq k+2$ whenever $k$ is odd. On the other hand, for every positive integer $r$, the following edge colouring of $G_{4r,3(r-1)}$  shows that $d(4r) \geq 3r-2$: for every even $1 \leq i \leq 2r$ and $1 \leq j \leq 3(r-1)$, let 
\begin{align*}
c(u_i v_{i,j}) & = j, \\
c(v_{i,j} u_{(i \bmod 2r) + 1}) & =
\begin{cases}
j + 1 \quad \text{ if } j \bmod 3 \neq 0, \\
j - 2 \quad \text{ if } j \bmod 3 = 0,
\end{cases}
\end{align*}
and colour the remaining edges arbitrarily such that the resulting colouring is proper. Similarly, one can show that $d(4r+1) \geq 3r-1$, $d(4r+2) \geq 3r-2$ and $d(4r+3) \geq 3r+2$.  

Note that the upper bounds on $AR_d(C_k)$ we proved in Theorem~\ref{thm:cycles} entail upper bounds on the size anti-Ramsey numbers of cycles. Since $G_{k,d}$ has $k d$ edges whenever $k$ is even and $(k-1) d + 1$ edges whenever $k$ is odd, it follows that $AR_s(C_k) \leq k d(k) \leq k(k-1)$ whenever $k$ is even and $AR_s(C_k) \leq (k-1) d(k) + 1 \leq (k-1)(k+2) + 1$ whenever $k$ is odd. Thus, even a slight improvement of the upper bound on $d(k)$ would improve the upper bound $AR_s(C_k) \leq (k-1)^2 + 1$ for even $k$ and $AR_s(C_k) \leq (k-1)^2$ for odd $k$ which was proved in~\cite{AKSU}.

\bigskip

\noindent \textbf{Degree anti-Ramsey numbers of complete bipartite graphs.} For all integers $1 \leq s \leq t$, let $f(s,t)$ denote the smallest integer $n$ such that any proper edge colouring of $K_{n,n}$ yields a rainbow copy of $K_{s,t}$. Similarly, for all integers $1 \leq s \leq t$, let $g(s,t)$ denote the smallest integer $n$ such that any proper edge colouring of $K_{s,n}$ yields a rainbow copy of $K_{s,t}$. Clearly 
\begin{equation} \label{eq::completeBipartiteBounds}
AR_d(K_{s,t}) \leq f(s,t) \leq g(s,t).
\end{equation}

Note that, trivially, $AR_d(K_{1,t}) = f(1,t) = g(1,t)$ for every positive integer $t$ and thus both inequalities in~\eqref{eq::completeBipartiteBounds} are sharp in general. It was proved in~\cite{CCDS} that $g(s,t) \leq (s^2 - s + 1)(t-1) + 1$ and that this simple upper bound is tight whenever $s-1$ is a non-negative prime power. It was also shown in~\cite{CCDS} that every proper edge colouring of $K_{3,6}$ yields a rainbow copy of $K_{2,3}$. It is then easy to verify that $f(2,3) = 6 < 7 = g(2,3)$. This is the only example we currently have where the second inequality in~\eqref{eq::completeBipartiteBounds} is strict. On the other hand, in addition to the trivial case $s=1$, we can prove that $f(2,t) = 3t-2$ holds whenever $t - 1$ is a non-negative power of $3$, and thus $f(2,t) = g(2,t)$ for those values of $t$. Indeed, for any positive integer $r$, let $n = 3^{r-1} + 1$ and consider the following edge colouring of $K_{3^r,3^r}$. Viewing each part of $K_{3^r,3^r}$ as the vector space $\mathbb{Z}_3^r$, we colour an edge $uv$ with the colour $u - v \in \mathbb{Z}_3^r$. This is a proper edge colouring of $K_{3^r,3^r}$ since, if $u, v_1, v_2 \in \mathbb{Z}_3^r$ are such that $u v_1$ and $u v_2$ have the same colour, then $u - v_1 = u - v_2$ and thus $v_1 = v_2$. Moreover, fix some distinct $u_1, u_2 \in \mathbb{Z}_3^r$ and pairwise distinct $v_1, \ldots, v_n \in \mathbb{Z}_3^r$. Since $3 n > 3^r$, the sets $\{v_i : 1 \leq i \leq n\}$, $\{v_i + u_1 - u_2 : 1 \leq i \leq n\}$ and $\{v_i + 2(u_1 - u_2) : 1 \leq i \leq n\}$ cannot be pairwise disjoint. Hence, there are $1 \leq i,j \leq n$ such that $u_1 - v_i = u_2 - v_j$ and thus the corresponding copy of $K_{2,n}$ is not rainbow.   
     
As for the first inequality in~\eqref{eq::completeBipartiteBounds}, we know it is sharp for $s=1$ and for $s = t = 2$ (by Proposition~\ref{propos:C4}). We are not aware of any $s \leq t$ for which $AR_d(K_{s,t}) < f(s,t)$. It would be interesting to further study the relations between these three functions.  
          
\bigskip

\noindent \textbf{Computability of degree anti-Ramsey numbers.} In order to determine the local anti-Ramsey number $AR(H)$ of a given graph $H$ we can, in principle, use the following algorithm:
\begin{enumerate}
\item Set $n$ to be $v(H)$.
\item If every proper edge colouring of $K_n$ (viewed as a partition of $E(K_n)$) yields a rainbow copy of $H$, then output $AR(H) = n$. Otherwise increase $n$ by $1$ and repeat step 2.
\end{enumerate}
Since there are only finitely many graphs with a prescribed number of edges, we can use a similar algorithm to compute the size anti-Ramsey number of any graph. On the other hand, we do not know how to devise an algorithm for computing degree anti-Ramsey numbers (note that the number of graphs with a prescribed maximum degree is unbounded). It would be interesting to find such an algorithm (or to prove that no such algorithm exists, though this seems unlikely). In order to obtain an algorithm for computing degree anti-Ramsey numbers in the spirit of the algorithms mentioned above, it would be sufficient to bound in some reasonable manner, for every graph $H$, the minimal number of vertices or edges in a graph with maximum degree $AR_d(H)$ such that any proper colouring of its edges yields a rainbow copy of $H$.

\bigskip

\noindent \textbf{Non-monotonicity in the number of colours.} It seems plausible that if any proper edge colouring of some graph $G$ with $m \geq \chi'(G)$ colours yields a rainbow copy of some graph $H$, then the same is true for any proper edge colouring of $G$ with $m' > m$ colours. However, this intuition fails in general. Indeed, consider the graph $G = (V,E)$, where $V = \{v_0,v_1,v_2,v_3,v_4,v_5,v_6,v_7,w\}$ and $E = \{v_k  v_{(k+1)\bmod 8} : 0 \leq k \leq 7\} \cup \{v_{2k} w : 0 \leq k \leq 3\} \cup \{v_0 v_4, v_2 v_6\}$. It is easy to verify that the chromatic index of $G$ is $4$, that every proper edge colouring of $G$ with $4$ colours yields a rainbow copy of $C_4$ and that $G$ does admit a proper edge colouring with no rainbow copy of $C_4$.

It would be interesting to find more examples of this phenomenon. In particular, we pose the following two questions. Is there an infinite family of graphs ${\mathcal H}$ such that for any $H \in {\mathcal H}$ there exists a graph $G$ such that any proper edge colouring of $G$ with $\chi'(G)$ colours yields a rainbow copy of $H$, but there exists a proper edge colouring of $G$ with no rainbow copy of $H$? Is there a non-negative integer $k$ such that for all graphs $H$ and $G$, if any proper edge colouring of $G$ with at most $\chi'(G) + k$ colours yields a rainbow copy of $H$, then any proper edge colouring of $G$ yields a rainbow copy of $H$?

\end{document}